\documentclass[a4paper]{amsart}
\usepackage{amssymb}
\usepackage[all]{xy}
\usepackage{hyperref, aliascnt}

\newtheorem{lma}{Lemma}[section]

\newaliascnt{thmCt}{lma}
\newtheorem{thm}[thmCt]{Theorem}
\aliascntresetthe{thmCt}

\newaliascnt{corCt}{lma}
\newtheorem{cor}[corCt]{Corollary}
\aliascntresetthe{corCt}

\newaliascnt{prpCt}{lma}
\newtheorem{prp}[prpCt]{Proposition}
\aliascntresetthe{prpCt}

\theoremstyle{definition}

\newaliascnt{pgrCt}{lma}
\newtheorem{pgr}[pgrCt]{}
\aliascntresetthe{pgrCt}

\newaliascnt{dfnCt}{lma}

\aliascntresetthe{dfnCt}

\newaliascnt{rmkCt}{lma}
\newtheorem{rmk}[rmkCt]{Remark}
\aliascntresetthe{rmkCt}

\newaliascnt{exaCt}{lma}
\newtheorem{exa}[exaCt]{Example}
\aliascntresetthe{exaCt}

\newaliascnt{qstCt}{lma}
\newtheorem{qst}[qstCt]{Question}
\aliascntresetthe{qstCt}

\newaliascnt{pbmCt}{lma}
\newtheorem{pbm}[pbmCt]{Problem}
\aliascntresetthe{pbmCt}

\def\today{\number\day\space\ifcase\month\or   January\or February\or
   March\or April\or May\or June\or   July\or August\or September\or
   October\or November\or December\fi\   \number\year}

\newcommand{\CC}{{\mathbb{C}}}

\newcommand{\ev}{{\mathrm{ev}}}
\newcommand{\tensProj}{\widehat{\otimes}}
\newcommand{\ca}{$C^*$-algebra}
\newcommand{\weakStar}{weak${}^*$}
\newcommand{\freeVar}{\_\,}
\newcommand{\Bdd}{{\mathcal{L}}}

\newcommand{\Cpct}{{\mathcal{K}}}
\newcommand{\Schatten}{{\mathcal{S}}}
\DeclareMathOperator{\tr}{tr}

\title[Preduals of $\Bdd(H,Y)$]{Preduals for spaces of operators involving Hilbert spaces and trace-class operators}
\date{\today}

\author{Hannes Thiel}
\address{Hannes Thiel. Mathematisches Institut, Universit\"at M\"unster, Einsteinst.~62, 48149 M\"unster, Germany.}
\email{hannes.thiel@uni-muenster.de}
\urladdr{www.math.uni-muenster.de/u/hannes.thiel/}

\thanks{The author was partially supported by the Deutsche Forschungsgemeinschaft (SFB 878 \emph{Groups, Geometry \& Actions}).}

\subjclass[2010]{Primary:
47L05, 
47L10, 
Secondary:
47L45, 
46B20. 
}
\keywords{Preduals, unique predual, Banach algebras, complemented subspace}

\begin{document}

\begin{abstract}
Continuing the study of preduals of spaces $\Bdd(H,Y)$ of bounded, linear maps, we consider the situation that $H$ is a Hilbert space.
We establish a natural correspondence between isometric preduals of $\Bdd(H,Y)$ and isometric preduals of $Y$.

The main ingredient is a Tomiyama-type result which shows that every contractive projection that complements $\Bdd(H,Y)$ in its bidual is automatically a right $\Bdd(H)$-module map.

As an application, we show that isometric preduals of $\Bdd(\Schatten_1)$, the algebra of operators on the space of trace-class operators, correspond to isometric preduals of $\Schatten_1$ itself (and there is an abundance of them).
On the other hand, the compact operators are the unique predual of $\Schatten_1$ making its multiplication separately \weakStar{} continuous.
\end{abstract}

\maketitle

\section{Introduction}

An isometric predual of a Banach space $X$ is a Banach space $F$ together with an isometric isomorphism $X\cong F^*$.
Every predual induces a \weakStar{} topology. 
Due to the importance of \weakStar{} topologies, it is interesting to study the existence and uniqueness of preduals;
see the survey \cite{God89IsometricPredualsSurvey} and the references therein.

Given Banach spaces $X$ and $Y$, let $\Bdd(X,Y)$ denote the space of operators from $X$ to $Y$.
Every isometric predual of $Y$ induces an isometric predual of $\Bdd(X,Y)$:
If $Y\cong F^*$, then $\Bdd(X,Y)\cong (X\hat{\otimes}F)^*$.

\begin{pbm}
\label{pbm:predual}
Find conditions on $X$ and $Y$ guaranteeing that every isometric predual of $\Bdd(X,Y)$ is induced from an isometric predual of $Y$.
\end{pbm}

Given reflexive spaces $X$ and $Y$, Godefroy and Saphar show that $X\tensProj Y^*$ is the strongly unique isometric predual of $\Bdd(X,Y)$; 
see \cite[Proposition~5.10]{GodSap88DualitySpOpsSmoothNorms}.
In particular, in this case every isometric predual $\Bdd(X,Y)$ is induced from $Y$. 

The main result of this paper extends this to the case that $X$ is a Hilbert space $H$ and $Y$ is arbitrary:
Every isometric predual of $\Bdd(H,Y)$ is induced from an isometric predual of $Y$;
see \autoref{prp:mainThm}.
In particular, $\Bdd(H,Y)$ has a (strongly unique) isometric predual if and only if $Y$ does;
see \autoref{prp:predualBHY}.

To obtain these results, we use that isometric preduals of $Y$ naturally correspond to contractive projections $\Bdd(X,Y)^{**}\to\Bdd(X,Y)$ that are right $\Bdd(X)$-module maps and have \weakStar{} closed kernel;
see \cite[Theorem~5.7]{GarThi16arX:PredualsBXY}.
Hence, we are faced with:

\begin{pbm}
\label{pbm:Tomiyama}
Find conditions on $X$ and $Y$ guaranteeing that every contractive projection $\Bdd(X,Y)^{**}\to\Bdd(X,Y)$ is automatically a right $\Bdd(X)$-module map.
\end{pbm}

It was shown by Tomiyama that every contractive projection from a \ca{} $A$ onto a sub-\ca{} $B$ is automatically a $B$-bimodule map.
We therefore consider a positive solution to \autoref{pbm:Tomiyama} a Tomiyama-type result.
Adapting the proof of Tomiyama's result, we obtain a positive solution to \autoref{pbm:Tomiyama} whenever $X$ is a Hilbert space;
see \autoref{prp:condExp}.

In \autoref{sec:predBS1Y}, we show that our results also hold when the Hilbert space $H$ is replaced by the space of trace-class operators $\Schatten_1(H)$.
It follows that isometric preduals of $\Bdd(\Schatten_1(H))$ naturally correspond to isometric preduals of $\Schatten_1(H)$;
see \autoref{exa:predualBS1}.
We note that $\Schatten_1(H)$ - and consequently $\Bdd(\Schatten_1(H))$ - has many different isometric preduals.
On the other hand, we show that the `standard' predual of compact operators is the unique predual of $\Schatten_1(H)$ making its multiplication separately \weakStar{} continuous;
see \autoref{prp:predualTraceClass}.

\subsection*{Acknowledgements}

The author would like to thank Eusebio Gardella and Tim de Laat for valuable feedback.

\subsection*{Notation}

Given Banach spaces $X$ and $Y$, an \emph{operator} from $X$ to $Y$ is a bounded, linear map $X\to Y$.
The space of such operators is denoted by $\Bdd(X,Y)$.
The projective tensor product of $X$ and $Y$ is denoted by $X\hat{\otimes}Y$.
We identify $X$ with a subspace of its bidual, and we let $\kappa_X\colon X\to X^{**}$ denote the inclusion.
A \emph{projection} $\pi\colon X^{**}\to X$ is an operator satisfying $\pi(x)=x$ for all $x\in X$.

\section{Preduals involving Hilbert spaces}
\label{sec:predBHY}

Throughout, $X$ and $Y$ denote Banach spaces.
For conceptual reasons, it is useful to consider preduals of $X$ as subsets of $X^*$.
More precisely, a closed subspace $F\subseteq X^*$ is an (isometric) predual of $X$ if for the inclusion map $\iota_F\colon F\to X^*$, the transpose map $\iota_F^*\colon X^{**}\to F^*$ restricts to an (isometric) isomorphism $X\to F^*$.

The space $X$ is said to have a \emph{strongly unique isometric predual} if there exists an isometric predual $F\subseteq X^*$ and if $F=G$ for every isometric predual $G\subseteq X^*$.
Every reflexive space $X$ has a strongly unique isometric predual, namely $X^*$.

\begin{pgr}
\label{pgr:actions}
The space $\Bdd(X,Y)$ has a natural $\Bdd(Y)$-$\Bdd(X)$-bimodule structure.
Given $a\in\Bdd(X)$, the action of $a$ is given by $R_a\colon\Bdd(X,Y)\to\Bdd(X,Y)$, $R_a(f):=f\circ a$, for $f\in\Bdd(X,Y)$.
Thus, $a$ acts by precomposing on the right of $\Bdd(X,Y)$.
Similarly, the action of $b\in\Bdd(Y)$ is given by postcomposing on the left, that is, by $L_b\colon\Bdd(X,Y)\to\Bdd(X,Y)$, $L_b(f):=b\circ f$, for $f\in\Bdd(X,Y)$.

We obtain a $\Bdd(X)$-$\Bdd(Y)$-bimodule structure on $\Bdd(X,Y)^*$.
The \emph{left} action of $a\in\Bdd(X)$ on $\Bdd(X,Y)^*$ is given by $R_a^*$.
The \emph{right} action of $b\in\Bdd(Y)$ on $\Bdd(X,Y)^*$ is given by $L_b^*$.
Similarly, we obtain a $\Bdd(Y)$-$\Bdd(X)$-bimodule structure on $\Bdd(X,Y)^{**}$. 
\end{pgr}

Given a \ca{} $A$ and $a,b,x,y\in A$ with $a^*b=0$, we have $\|ax+by\|^2\leq\|ax\|^2+\|by\|^2$, which is an analog of Bessel's inequality;
see \cite[II.3.1.12, p.66]{Bla06OpAlgs}.
We first prove two versions of this result in a more general context.

\begin{lma}
\label{prp:BesselLeft}
Let $H$ be a Hilbert space, let $X$ be a Banach space, let $a,b\in\Bdd(H)$ satisfying $a^*b=0$, and let $f,g\in\Bdd(X,H)$.
Then
\[
\|af + bg\|^2 \leq \|af\|^2 + \|bg\|^2.
\]
\end{lma}
\begin{proof}
The equation $a^*b=0$ implies that the ranges of $a$ and $b$ are orthogonal.
Given $x\in X$, it follows that the elements $afx$ and $bgx$ are orthogonal in $H$, whence
\[
\|afx + bgx\|^2 = \|afx\|^2 + \|bgx\|^2.
\]
Using this at the second step, we deduce that
\[
\|af + bg\|^2
= \sup_{\|x\|\leq 1} \|afx + bgx\|^2
= \sup_{\|x\|\leq 1} \big( \|afx\|^2 + \|bgx\|^2 \big)
\leq \|af\|^2 + \|bg\|^2,
\]
as desired.
\end{proof}

\begin{lma}
\label{prp:BesselRight}
Let $H$ be a Hilbert space, let $Y$ be a Banach space, let $a,b\in\Bdd(H)$ satisfying $ab^*=0$, and let $f,g\in\Bdd(H,Y)$.
Then
\[
\|fa + gb\|^2 \leq \|fa\|^2 + \|gb\|^2.
\]

Similarly, given $F,G\in\Bdd(H,Y)^{**}$, we have $\|Fa + Gb\|^2 \leq \|Fa\|^2 + \|Gb\|^2$.
\end{lma}
\begin{proof}
We denote the transpose of an operator $h$ by $h^t$, to distinguish it from the adjoint of an operator in $\Bdd(H)$.
We have $a^t,b^t\in\Bdd(H^*)$ and $f^t,g^t\in\Bdd(Y^*,H^*)$.
Further, $(fa)^t=a^tf^t$, where $a^tf^t$ is given by the left action of $\Bdd(H^*)$ on $\Bdd(Y^*,H^*)$.
It follows from $ab^*=0$ that $ba^*=0$.
We have $(a^*)^t=(a^t)^*$ in $\Bdd(H^*)$, and therefore
\[
(a^t)^*b^t = (a^*)^tb^t=(ba^*)^t=0.
\]

Applying \autoref{prp:BesselLeft} at the third step, we compute
\[
\|fa + gb\|^2
= \|(fa + gb)^t\|^2
= \|a^tf^t+b^tg^t\|^2
\leq \|a^tf^t\|^2 + \|b^tg^t\|^2
= \|fa\|^2 + \|gb\|^2.
\]

Let us show the second inequality.
Let $p$ and $q$ be the right support projections of $a$ and $b$ in $\Bdd(H)$, respectively.
Then $ap=a$, $bq=b$, and $pq^*=0$.
It follows that $Fa=Fap$ and $Gb=Gbq$.

Using Goldstine's theorem, we choose nets $(f_i)_i$ and $(g_j)_j$ in $\Bdd(H,Y)$ such that $(f_i)_i$ converges \weakStar{} to $Fa$, such that $(g_j)_j$ converges \weakStar{} to $Gb$, and such that $\|f_i\|\leq\|Fa\|$ for all $i$ and $\|g_j\|\leq\|Gb\|$ for all $j$.
Then $(f_ip)_i$ converges \weakStar{} to $Fap$.
Using this at the second step, we deduce that
\[
\|Fa\| = \|Fap\| \leq \varliminf_i \|f_ip\| \leq \varlimsup_i \|f_ip\| \leq \|Fa\|,
\]
and hence $\lim_i \|f_ip\| = \|Fa\|$.
Analogously, we obtain that $\lim_j \|g_jq\| = \|Gb\|$.

Using this at the third step, using that the net $(f_ip+g_jq)_{i,j}$ converges \weakStar{} to $Fa+Gb$ at the first step, and using the first inequality of this lemma at the second step, we deduce that
\[
\|Fa + Gb\|^2
\leq \varliminf_{i,j} \|f_ip + g_jq\|^2
\leq \varliminf_{i,j} \left( \|f_ip\|^2 + \|g_jq\|^2 \right)
= \|Fa\|^2 + \|Gb\|^2. \qedhere
\]
\end{proof}

Let $A$ be a \ca{} and let $B\subseteq A$ be a sub-\ca{}.
By Tomiyama's theorem, every contractive projection $\pi\colon A\to B$ is automatically a $B$-bimodule map (called a conditional expectation). 
The next result is in the same spirit. 
It provides a partial positive solution to \autoref{pbm:Tomiyama}.

\begin{thm}
\label{prp:condExp}
Let $H$ be a Hilbert space, and let $Y$ be a Banach space.
Then every contractive projection $\pi\colon\Bdd(H,Y)^{**}\to\Bdd(H,Y)$ is automatically a right $\Bdd(H)$-module map, that is, $\pi(Fa)=\pi(F)a$ for every $F\in\Bdd(H,Y)^{**}$ and $a\in\Bdd(H)$.
\end{thm}
\begin{proof}
First, we show the result for the case that $a$ is a projection.
Let $p\in\Bdd(H)$ satisfy $p=p^2=p^*$, and set $q:=1-p$.
The following argument is adapted from the proof of Tomiyama's theorem in \cite[Theorem~II.6.10.2, p.132]{Bla06OpAlgs}.
Let $\lambda>0$.
We have $\pi( \pi(Fp)q )q= \pi(Fp)q$.
Using this at the first step, using that $\|\pi\|\leq 1$ and $\|q\|\leq 1$ at the third step, and using \autoref{prp:BesselRight} at the fourth step, we get
\begin{align*}
(1+\lambda)^2\|\pi(Fp)q\|^2
&= \left\| \pi(Fp)q + \lambda \pi\big( \pi(Fp)q \big)q \right\|^2
= \left\| \pi\big( Fp + \lambda \pi(Fp)q \big)q  \right\|^2 \\
&\leq \| Fp + \lambda \pi(Fp)q \|^2
\leq \|Fp\|^2  + \| \lambda \pi(Fp)q \|^2 \\
&= \|Fp\|^2  + \lambda^2 \| \pi(Fp)q \|^2.
\end{align*}

It follows that
\[
(1+2\lambda)\|\pi(Fp)q\|^2
\leq \| Fp \|^2.
\]

Since this holds for every $\lambda>0$, we deduce that $\pi(Fp)q=0$.
Adding $\pi(Fp)p$ to this equation, we obtain
\[
\pi(Fp)=\pi(Fp)p.
\]

Switching the place of $p$ and $q$ in the above argument, we get $\pi(Fq)p=0$.
Adding $\pi(Fp)p$, we get $\pi(F)p=\pi(Fp)p$.
We deduce that
\[
\pi(Fp)=\pi(Fp)p=\pi(F)p.
\]

Finally, we use that every operator on a Hilbert space is a linear combination of finitely many projections;
see \cite[Corollary~2.3]{PeaTop67SumsIdemp}.
\end{proof}

\begin{pgr}
\label{pgr:alpha}
Given $x\in X$, we let $\ev_x\colon\Bdd(X,Y)\to Y$ denote the evaluation map, given by $\ev_x(f):=f(x)$, for $f\in\Bdd(X,Y)$.
Let $\alpha_{X,Y}\colon\Bdd(X,Y)^{**}\to\Bdd(X,Y^{**})$ be the operator introduced in \cite[Definition~3.16]{GarThi16arX:PredualsBXY}.
We have
\[
\alpha_{X,Y}(F)(x)=\ev_x^{**}(F),
\]
for all $F\in\Bdd(X,Y)^{**}$ and $x\in X$;
see \cite[Lemma~3.18]{GarThi16arX:PredualsBXY}.
The map $\alpha_{X,Y}$ is always a contractive, right $\Bdd(X)$-module map.

Let $\pi\colon Y^{**}\to Y$ be a projection. 
Define $\pi_*\colon\Bdd(X,Y^{**})\to\Bdd(X,Y)$ by $\pi_*(f):=\pi\circ f$, for $f\in\Bdd(X,Y^{**})$.
Set $r_\pi:=\pi_*\circ\alpha_{X,Y}\colon\Bdd(X,Y)^{**}\to\Bdd(X,Y)$.
Note that
\[
r_\pi(F)(x) = \pi(\ev_x^{**}(F)),
\]
for $F\in\Bdd(X,Y)^{**}$ and $x\in X$.
The map $r_\pi$ was considered in \cite[Section~4]{GarThi16arX:PredualsBXY}, where it is shown that $r_\pi$ is a right $\Bdd(X)$-module projection.
\end{pgr}

Recall that (concrete) isometric preduals of $Y$ are in natural bijection with contractive projections $Y^{**}\to Y$ that have \weakStar{} closed kernel;
see for example \cite[Proposition~2.7]{GarThi16arX:PredualsBXY}.
Every predual induces a \weakStar{} topology.
A predual of $\Bdd(X,Y)$ makes the right action by each element from $\Bdd(X)$ \weakStar{} continuous if and only if the associated projection $\Bdd(X,Y)^{**}\to\Bdd(X,Y)$ is a right $\Bdd(X)$-module map;
see for example \cite[Proposition~B.6]{GarThi16arX:PredualsBXY}.
The following result is contained in \cite[Theorem~4.7]{GarThi16arX:PredualsBXY}.

\begin{thm}
\label{prp:corrProjYProjBXY}
Let $X$ and $Y$ be Banach spaces with $X\neq\{0\}$.
Given a projection $\pi\colon Y^{**}\to Y$, let $r_\pi\colon\Bdd(X,Y)^{**}\to\Bdd(X,Y)$ be as in \autoref{pgr:alpha}.
This assignment defines a natural bijection between the following classes:
\begin{enumerate}
\item
Contractive projections $Y^{**}\to Y$;
\item
Contractive, right $\Bdd(X)$-module projections $\Bdd(X,Y)^{**}\to\Bdd(X,Y)$.
\end{enumerate}
Moreover, the kernel of $\pi$ is \weakStar{} closed if and only if the kernel of $r_\pi$ is.
Thus, the above correspondence restricts to a natural bijection between isometric preduals of $Y$ and isometric preduals of $\Bdd(X,Y)$ that make the right action by $\Bdd(X)$ \weakStar{} continuous.
\end{thm}

Combining Theorems~\ref{prp:condExp} and~\ref{prp:corrProjYProjBXY}, we obtain the main result of this paper:

\begin{thm}
\label{prp:mainThm}
Let $H$ be a Hilbert space with $H\neq\{0\}$, and let $Y$ be a Banach space.
Assigning to a projection $\pi\colon Y^{**}\to Y$ the projection $r_\pi\colon\Bdd(H,Y)^{**}\to\Bdd(H,Y)$, as in \autoref{prp:corrProjYProjBXY}, establishes a natural bijection between the following classes.
\begin{enumerate}
\item
Contractive projections $Y^{**}\to Y$;
\item
Contractive projections $\Bdd(H,Y)^{**}\to\Bdd(H,Y)$.
\end{enumerate}

Restricted to projections with \weakStar{} closed kernel, we obtain a natural bijection between isometric preduals of $Y$ and isometric preduals of $\Bdd(H,Y)$.
\end{thm}

\begin{cor}
\label{prp:predualBHY}
Let $H$ be a Hilbert space with $H\neq\{0\}$, and let $Y$ be a Banach space.
Then $\Bdd(H,Y)$ is $1$-complemented in its bidual if and only if $Y$ is.
Further, $\Bdd(H,Y)$ has an isometric predual if and only if $Y$ does.
Moreover, if $\Bdd(H,Y)$ has a strongly unique isometric predual if and only if $Y$ does.
\end{cor}

\begin{cor}
\label{prp:predHilbWCts}
Let $H$ be a Hilbert space, let $Y$ be a Banach space, and let $F\subseteq\Bdd(H,Y)^*$ be an isometric predual.
Then for each $a\in\Bdd(H)$, the right action of $a$ on $\Bdd(H,Y)$ is continuous for the \weakStar{} topology induced by $F$.
\end{cor}

\begin{rmk}
There is a canonical isometric isomorphism $(X\hat{\otimes}Y)^* \cong \Bdd(X,Y^*)$.
Given an isometric isomorphism $Y\cong F^*$, we obtain isometric isomorphisms
\[
\Bdd(X,Y) \cong \Bdd(X,F^*) \cong (X\tensProj F)^*.
\]
Hence, $X\tensProj F$ is an isometric predual of $\Bdd(X,Y)$.
Given a Hilbert space $H$, \autoref{prp:mainThm} states that every isometric predual of $\Bdd(H,Y)$ occurs this way.
In particular, given an isometric isomorphism $\Bdd(H,Y)\cong G^*$ for some Banach space $G$, there is an isometric isomorphism $H\hat{\otimes} F \cong G$, for some isometric predual $F$ of $Y$.
\end{rmk}

\begin{rmk}
By \cite[Proposition~5.10]{GodSap88DualitySpOpsSmoothNorms}, $X\hat{\otimes}F$ is the strongly unique isometric predual of $\Bdd(X,F^*)$ if $X$ and $F$ satisfy the Radon Nikod\'{y}m property (RNP).
Every reflexive space (in particular, every Hilbert space) satisfies the RNP.
Thus, if $Y$ has an isometric predual satisfying the RNP, then \autoref{prp:predualBHY} follows from \cite{GodSap88DualitySpOpsSmoothNorms}.
However, not every Banach spaces with strongly unique isometric predual occurs as the dual of a space satisfying the RNP;
see \autoref{exa:predualBHM}.
\end{rmk}

\begin{exa}
\label{exa:predualBHM}
Let $M$ be a von Neumann algebra.
By Sakai's theorem, $M$ has a strongly unique isometric predual, denoted by $M_*$.
It follows from \autoref{prp:mainThm} that $H\hat{\otimes}M_*$ is the strongly unique isometric predual of $\Bdd(H,M)$.
By \cite[Theorem~4]{Chu81ScatteredRNP}, $M_*$ satisfies the RNP if and only if $M$ is a direct sum of type~$\mathrm{I}$ factors.

Let $\mathcal{R}$ denote the hyperfinite $\mathrm{II}_1$-factor.
Then $\mathcal{R}_*$ does not have the RNP, yet $H\hat{\otimes}\mathcal{R}_*$ is the strongly unique isometric predual of $\Bdd(H,\mathcal{R})$.
\end{exa}

\begin{qst}
\label{qst:analog}
Does \autoref{prp:mainThm} hold when the Hilbert space is replaced by a general Banach space satisfying the RNP?
More modestly, if $X$ is an $L^p$-space, do isometric preduals of $\Bdd(X,Y)$ correspond to isometric preduals of $Y$?
\end{qst}

\begin{rmk}
\label{rmk:Tomiyama}
Note that every positive solution of \autoref{pbm:Tomiyama} leads to an analog of \autoref{prp:mainThm}.
Therefore, \autoref{qst:analog} has a positive answer if the following instance of \autoref{pbm:Tomiyama} has a positive solution:
Given a measure space $\mu$ and a Banach space $Y$, is every contractive projection $\Bdd(L^p(\mu),Y)^{**}\to\Bdd(L^p(\mu),Y)$ automatically a right $\Bdd(L^p(\mu))$-module map?

Consider the space $\ell^\infty$ of bounded sequences.
Since $\ell^\infty$ is a von Neumann algebra, it has a strongly unique isometric predual.
Thus, if \autoref{pbm:Tomiyama} had a positive solution for $X=Y=\ell^\infty$, then $\Bdd(\ell^\infty)$ would have a strongly unique isometric predual.
However, it was noted in \cite[Remark~5.12]{GodSap88DualitySpOpsSmoothNorms} that $\Bdd(\ell^\infty)$ has nonisomorphic isometric preduals.
It follows in particular that there exists a contractive projection $\Bdd(\ell^\infty)^{**}\to\Bdd(\ell^\infty)$ that is not a right $\Bdd(\ell^\infty)$-module map.
\end{rmk}

\section{Preduals involving trace-class operators}
\label{sec:predBS1Y}

Throughout this section $H$ denotes a Hilbert space with $H\neq\{0\}$.
We let $\Cpct(H)$ and $\Schatten_1(H)$ denote the compact and trace-class operators on $H$, respectively.

An operator $a\in\Bdd(H)$ belongs to $\Schatten_1(H)$ if and only if for some (equivalently, every) orthonormal basis $(e_j)_j$ of $H$ the sum $\sum_j \langle |a|e_j,e_j\rangle$ is finite.
Given $a\in\Schatten_1(H)$ and an orthonormal basis $(e_j)_j$ of $H$, the sum $\sum_j\langle ae_j,e_j\rangle$ converges absolutely.
Moreover, it is independent of the choice of a orthonormal basis and we call
\[
\tr(a) := \sum_j\langle ae_j,e_j\rangle
\]
the \emph{trace} of $a$.
We set $\|a\|_1 := \tr(|a|)$.
This defines a norm on $\Schatten_1(H)$, turning the trace-class operators into a Banach space.
Note that $\Schatten_1(H)$ is a (non-closed) two-sided ideal in $\Bdd(H)$.
Moreover, we have $\Schatten_1(H)\subseteq\Cpct(H)$.

\begin{pgr}
\label{pgr:dualTraceClass}
Given $a\in\Bdd(H)$, the map $\Schatten_1(H)\to\CC$, $x\mapsto\tr(ax)$, is a bounded, linear functional on $\Schatten_1(H)$.
This induces an isometric isomorphism $\Bdd(H) \cong \Schatten_1(H)^*$. 
It is also well known that $\Cpct(H)\subseteq\Bdd(H)=\Schatten_1(H)^*$ is an isometric predual of $\Schatten_1(H)$.
Thus, we have isometric isomorphisms
\[
\Cpct(H)^* \cong \Schatten_1(H),\quad\text{ and }\quad \Schatten_1(H)^* \cong \Bdd(H).
\]
\end{pgr}

Let us say that a Banach space $X$ satisfies ($*$) if for every Banach space $E$, every contractive projection $\Bdd(X,E)^{**}\to\Bdd(X,E)$ is automatically a right $\Bdd(X)$-module map.
If $X$ satisfies ($*$), then it follows from \autoref{prp:corrProjYProjBXY} that every contractive projection $\Bdd(X,E)^{**}\to\Bdd(X,E)$ is of the form $\pi_\ast\circ\alpha_{X,E}$ for a unique contractive projection $\pi\colon E^{**}\to E$.
\autoref{prp:condExp} states that Hilbert spaces satisfy ($*$).

\begin{lma}
\label{prp:TomiyamaTensProd}
If two Banach spaces $X$ and $Y$ satisfy ($*$), then so does $X\hat{\otimes}Y$.
\end{lma}
\begin{proof}
Assume that $X$ and $Y$ satisfy ($*$).
Let $E$ be another Banach space, and let $\pi\colon\Bdd(X\hat{\otimes}Y,E)^{**}\to\Bdd(X\hat{\otimes}Y,E)$ be a contractive projection.

Banach spaces form a closed monoidal category for the projective tensor product and with $\Bdd(Y,\freeVar)$ adjoint to $\freeVar\hat{\otimes}Y$.
Thus, there is a natural isometric isomorphism
\[
\Bdd\big( X\hat{\otimes}Y,E \big)
\cong \Bdd\big( X,\Bdd(Y,E) \big).
\]
An operator $f\colon X\hat{\otimes}Y\to E$ is identified with the operator $\tilde{f}\colon X\to\Bdd(Y,E)$ that sends $x\in X$ to the operator $\tilde{f}(x)\colon Y\to E$ given $\tilde{f}(x)(y) := f(x\otimes y)$, for $y\in Y$.

The projection $\pi$ corresponds to a contractive projection $\pi'\colon\Bdd(X,\Bdd(Y,E))^{**}\to\Bdd(X,\Bdd(Y,E))$.
Applying that $X$ satisfies $(*)$ to the projection $\pi'$, there exists a unique contractive projection $\tau\colon\Bdd(Y,E)^{**}\to\Bdd(Y,E)$ such that $\pi'=\tau_*\circ\alpha_{X,\Bdd(Y,E)}$.
The situation is shown in the following diagram:
\[
\xymatrix@C+10pt{
{ \Bdd\big( X\hat{\otimes}Y, E \big)^{**} }
\ar@/^1pc/[d]^{\pi}
\ar@{}[r]|-{\cong}
& { \Bdd\big( X, \Bdd(Y, E) \big)^{**} }
\ar@/^1pc/[d]^{\pi'}
\ar[r]^{\alpha_{X,\Bdd(Y,E)}}
& { \Bdd\big( X, \Bdd(Y,E)^{**} \big) }
\ar[dl]^{\tau_*}
\\
{ \Bdd\big( X\hat{\otimes}Y, E \big) }
\ar@{^{(}->}[u]
\ar@{}[r]|-{\cong}
& { \Bdd\big( X, \Bdd(Y, E) \big) }
\ar@{^{(}->}[u]
& {\makebox[75pt]{}.}
}
\]
Let $F\in\Bdd\big( X\hat{\otimes}Y, E \big)^{**}$, which corresponds to $\tilde{F}\in\Bdd\big( X, \Bdd(Y, E) \big)^{**}$.
Then
\[
\pi(F)(x\otimes y)
= \big[ \pi'(\tilde{F})(x) \big](y)
= \tau\big( \ev_x^{**}(\tilde{F}) \big) (y),
\]
for $x\in X$ and $y\in Y$.

Applying that $Y$ satisfies $(*)$ to the projection $\tau$, there exists a unique contractive projection $\sigma\colon E^{**}\to E$ such that
\[
\tau(G)(y) = \sigma\big( \ev_y^{**}(G) \big),
\]
for every $G\in\Bdd(Y,E)^{**}$ and $y\in Y$.
We claim that $\sigma$ is the desired projection to verify that $X\hat{\otimes}Y$ satisfies $(*)$.

Given $f\in\Bdd(X\hat{\otimes}Y,E)$ with corresponding element $\tilde{f}\in\Bdd(X,\Bdd(Y,E))$, note that $\ev_{x\otimes y}(f)=\ev_y(\ev_x(\tilde{f}))$.
It follows that $\ev_{x\otimes y}^{**}(F)=\ev_y^{**}(\ev_x^{**}(\tilde{F}))$.
Using this at the last step, we deduce that
\[
\pi(F)(x\otimes y)
= \tau\big( \ev_x^{**}(\tilde{F}) \big) (y)
= \sigma\big( \ev_y^{**}(\ev_x^{**}(\tilde{F})) \big)
= \sigma\big( \ev_{x\otimes y}^{**}(F) \big),
\]
for every $x\in X$ and $y\in Y$.
Thus, we have $\pi(F)(t)= \sigma(\ev_t^{**}(F) )$, for every simple tensor $t$ in $X\hat{\otimes}Y$.
It follows from linearity and continuity of the involved maps that the same equation holds for every $t\in X\hat{\otimes}Y$, as desired.
\end{proof}

\begin{lma}
\label{prp:condExpTraceClass}
Let $Y$ be a Banach space.
Then every contractive projection from $\Bdd(\Schatten_1(H),Y)^{**}$ to $\Bdd(\Schatten_1(H),Y)$ is automatically a right $\Bdd(\Schatten_1(H))$-module map.
\end{lma}
\begin{proof}
It is well known that the trace-class operators on $H$ are isometrically isomorphic to $H\hat{\otimes} H$.
Therefore, the statement follows from \autoref{prp:TomiyamaTensProd}.
\end{proof}

Using \autoref{prp:condExpTraceClass} instead of \autoref{prp:condExp}, we obtain the analog of \autoref{prp:mainThm} with the space of trace-class operators in place of the Hilbert space.

\begin{prp}
\label{prp:mainThmTraceClass}
Let $Y$ be a Banach space.
Assigning to a projection $\pi\colon Y^{**}\to Y$ the projection $r_\pi\colon\Bdd(\Schatten_1,Y)^{**}\to\Bdd(\Schatten_1,Y)$, as in \autoref{prp:corrProjYProjBXY}, established a natural bijection between the following classes.
\begin{enumerate}
\item
Contractive projections $Y^{**}\to Y$;
\item
Contractive projections $\Bdd(\Schatten_1(H),Y)^{**}\to\Bdd(\Schatten_1(H),Y)$.
\end{enumerate}

Restricted to projections with \weakStar{} closed kernel, we obtain a natural bijection between isometric preduals of $Y$ and isometric preduals of $\Bdd(\Schatten_1(H),Y)$.
\end{prp}

\begin{cor}
\label{prp:predTraceClassWCts}
Let $Y$ be a Banach space, and let $F\subseteq\Bdd(\Schatten_1(H),Y)^*$ be an isometric predual.
Then for each $a\in\Bdd(\Schatten_1(H))$, the right action of $a$ on $\Bdd(\Schatten_1(H),Y)$ is continuous for the \weakStar{} topology induced by $F$.
\end{cor}

\begin{rmk}
Using \autoref{prp:TomiyamaTensProd} inductively, \autoref{prp:mainThmTraceClass} and \autoref{prp:predTraceClassWCts} hold for any projective tensor power $H\hat{\otimes}\ldots\hat{\otimes}H$ in place of $\Schatten_1(H)$.
\end{rmk}

\begin{rmk}
We have isometric isomorphisms
\[
\Bdd(\Schatten_1(H))
\cong \Bdd\big( \Schatten_1(H), \Cpct(H)^* \big)
\cong \left( \Schatten_1(H)\hat{\otimes}\Cpct(H) \right)^*,
\]
and we consider $\Schatten_1(H)\hat{\otimes}\Cpct(H)$ as the `standard' predual of $\Bdd(\Schatten_1(H))$.
Right multiplication by an element from $\Bdd(\Schatten_1(H))$ is continuous for the induced \weakStar{} topology;
see \autoref{prp:predTraceClassWCts}.
However, this is not the case for left multiplication.
Indeed, a Banach space $X$ is reflexive if and only if left multiplication by elements from $\Bdd(X)$ is \weakStar{} continuous (for any predual);
see \cite[Corollary~6.4]{GarThi16arX:PredualsBXY}.
\end{rmk}

\begin{exa}
\label{exa:predualBS1}
By \autoref{prp:mainThmTraceClass}, there is a natural correspondence between isometric preduals of $\Bdd(\Schatten_1(H))$ and isometric preduals of $\Schatten_1(H)$.
The compact operators on $H$ form the canonical isometric predual of $\Schatten_1(H)$.

However, if $H=\ell_2$, then $\Schatten_1(H)$ has also many other isometric preduals.
Indeed, the diagonal operators in $\Schatten_1(H)$ form an isometric copy of $\ell_1$ that is closed for the `standard' \weakStar{} topology induced by the compact operators.
Since $\ell_1$ does not have a strongly unique isometric predual, neither does $\Schatten_1(H)$.
\end{exa}

\begin{thm}
\label{prp:predualTraceClass}
The `standard' predual $\Cpct(H)$ of $\Schatten_1(H)$ makes multiplication in $\Schatten_1(H)$ separately \weakStar{} continuous.
Moreover, $\Cpct(H)$ is the only such predual:
If $F\subseteq \Schatten_1(H)^*=\Bdd(H)$ is a (not necessarily isometric) predual that makes multiplication in $\Schatten_1(H)$ separately \weakStar{} continuous, then $F=\Cpct(H)$.
In particular, every predual of $\Schatten_1(H)$ making multiplication separately \weakStar{} continuous is automatically an isometric predual.
\end{thm}
\begin{proof}
Given $a\in\Bdd(H)$, we let $\varphi_a$ denote the functional on $\Schatten_1(H)$ given by $\langle\varphi_a,x\rangle := \tr(ax)$, for $x\in \Schatten_1(H)$.
This identifies $\Schatten_1(H)^*$ with $\Bdd(H)$.

The multiplication on $\Schatten_1(H)$ induces a $\Schatten_1(H)$-bimodule structure on its dual;
see \cite[Paragraph~A.3]{GarThi16arX:PredualsBXY}.
Let us recall some details.
Given $a\in \Schatten_1(H)$, let $L_a,R_a\colon \Schatten_1(H)\to \Schatten_1(H)$ be given by $L_a(x):=ax$ and $R_a(x):=xa$, for $x\in \Schatten_1(H)$.
Then the left action of $a$ on $\Schatten_1(H)^*$ is given by $R_a^*$, and the right action is given by $L_a^*$.
Thus, given $a\in \Schatten_1(H)$ and $b\in\Bdd(H)$, we have
\[
\langle \varphi_b a, x \rangle
=  \langle L_a^*(\varphi_b), x \rangle
=  \langle \varphi_b, ax \rangle
= \tr(bax)
= \langle \varphi_{ba}, x \rangle,
\]
for $x\in \Schatten_1(H)$, and therefore $\varphi_ba=\varphi_{ba}$.
Similarly, we obtain $a\varphi_b=\varphi_{ab}$.

Let $F\subseteq \Schatten_1(H)^*=\Bdd(H)$ be a predual.
Then left (right) multiplication on $\Schatten_1(H)$ is $\sigma(\Schatten_1(H),F)$-continuous if and only if $F$ is a right (left) $\Schatten_1(H)$-submodule of $\Bdd(H)$;
see \cite[Corollary~B.7]{GarThi16arX:PredualsBXY}.
Given $a\in\Cpct(H)$ and $b\in \Schatten_1(H)$, we have $ab,ba\in\Cpct(H)$.
Thus, the predual $\Cpct(H)$ is a $\Schatten_1(H)$-sub-bimodule of $\Bdd(H)$, which shows that it makes multiplication in $\Schatten_1(H)$ separately \weakStar{} continuous.

Conversely, let $F\subseteq \Schatten_1(H)^*=\Bdd(H)$ be a predual making multiplication in $\Schatten_1(H)$ separately \weakStar{} continuous.
Then $F$ is invariant under the left and right action of $\Schatten_1(H)$ on $\Bdd(H)$.
We have shown above that the left (right) action of $a\in \Schatten_1(H)$ on $\Bdd(H)$ is simply given by right (left) multiplication with $a$.

Claim: The set $F\cap\Cpct(H)$ is a closed, two-sided ideal in $\Cpct(H)$.
To verify the claim, let $a\in F\cap\Cpct(H)$, and let $b\in\Cpct(H)$.
Given a finite-dimensional subspace $D\subseteq H$, let $p_D$ be the orthogonal projection onto $D$.
We order the finite-dimensional subspaces of $H$ by inclusion.
Since $b\in\Cpct(H)$, we have $\lim_D \|p_D b - b \|=0$ and therefore
\[
\lim_D \|ap_Db - ab\| = 0.
\]
For each $D$, we have $p_Db\in \Schatten_1(H)$.
Since $F$ is invariant under right multiplication by $\Schatten_1(H)$, it follows that $ap_Db\in F\cap\Cpct(H)$.
Since $F$ is norm-closed, we deduce that $ab\in F\cap\Cpct(H)$.
Analogously, one shows that $F\cap\Cpct(H)$ is a left ideal in $\Cpct(H)$, which proves the claim.

The only closed, two-sided ideals of $\Cpct(H)$ are $\{0\}$ and $\Cpct(H)$.
It is easy to see that $F\cap\Cpct(H)\neq\{0\}$.
Thus, $\Cpct(H)\subseteq F$.
Since both $\Cpct(H)$ and $F$ are preduals of $\Schatten_1(H)$, it follows that $\Cpct(H)=F$, as desired.
\end{proof}

\begin{cor}
Let $\Schatten_1(H)$ be the trace-class operators on a Hilbert space $H$.
Then every Banach algebra isomorphism $\Schatten_1(H)\to \Schatten_1(H)$ is \weakStar{} continuous (for the `standard' predual $\Cpct(H)$.)
\end{cor}

\begin{rmk}
A \emph{dual Banach algebra} is a Banach algebra $A$ together with a predual $F\subseteq A^*$ making the multiplication in $A$ separately \weakStar{} continuous.
This concept was introduced by Runde, \cite[Definition~4.4.1, p.108]{Run02BookAmen}, and extensively studied by Daws;
see \cite{Daw11BicommDualBAlg} and the references therein.
\autoref{prp:predualTraceClass} states that the trace-class operators with their `standard' predual of compact operators form a dual Banach algebra.
Moreover, the compact operators are the only predual making the trace-class operators into a dual Banach algebra.
\end{rmk}


\providecommand{\bysame}{\leavevmode\hbox to3em{\hrulefill}\thinspace}
\providecommand{\noopsort}[1]{}
\providecommand{\mr}[1]{\href{http://www.ams.org/mathscinet-getitem?mr=#1}{MR~#1}}
\providecommand{\zbl}[1]{\href{http://www.zentralblatt-math.org/zmath/en/search/?q=an:#1}{Zbl~#1}}
\providecommand{\jfm}[1]{\href{http://www.emis.de/cgi-bin/JFM-item?#1}{JFM~#1}}
\providecommand{\arxiv}[1]{\href{http://www.arxiv.org/abs/#1}{arXiv~#1}}
\providecommand{\doi}[1]{\url{http://dx.doi.org/#1}}
\providecommand{\MR}{\relax\ifhmode\unskip\space\fi MR }
\providecommand{\MRhref}[2]{%
  \href{http://www.ams.org/mathscinet-getitem?mr=#1}{#2}
}
\providecommand{\href}[2]{#2}

\end{document}